\documentclass[a4paper, 12pt, oneside]{amsart}

\usepackage{mathptmx}       
\usepackage{helvet}         
\usepackage{courier}        
\usepackage{type1cm}        
\usepackage{graphicx}        
\usepackage[bottom]{footmisc}

\usepackage{amssymb, amsmath}
\usepackage{xcolor}
\usepackage{calc}
\usepackage{enumitem}
\usepackage[DIV=10]{typearea}

\theoremstyle{theorem}
\newtheorem{thm}{Theorem}[section]
\newtheorem{lem}[thm]{Lemma}

\newtheorem{prp}[thm]{Proposition}

\theoremstyle{definition}

\newtheorem{dfn}[thm]{Definition}
\newtheorem{rmk}[thm]{Remark}
\newtheorem*{rmk*}{Remark}

\newcommand{\CC}{\mathbb{C}}
\newcommand{\NN}{\mathbb{N}}

\newcommand{\TT}{\mathbb{T}}
\newcommand{\ZZ}{\mathbb{Z}}

\newcommand{\Bb}{\mathcal{B}}
\newcommand{\Cc}{\mathcal{C}}

\newcommand{\Hh}{\mathcal{H}}
\newcommand{\Kk}{\mathcal{K}}
\newcommand{\Ii}{\mathcal{I}}

\newcommand{\Tt}{\mathcal{T}}

\newcommand{\clsp}{\overline{\operatorname{span}}}

\newcommand{\Per}{\operatorname{Per}}
\newcommand{\Prim}{\operatorname{Prim}}
\newcommand{\spec}{\operatorname{spec}}
\newcommand{\CH}{\Cc(H)}

\title[Ideals of graph algebras]{On Hong and Szyma\'nski's description of the primitive-ideal space of a graph algebra}

\author{Toke Meier Carlsen}
\address[T.M. Carlsen]{Department of Science and Technology\\University of the Faroe Islands\\
N\'oat\'un 3\\ FO-100 T\'orshavn\\the Faroe Islands}
\email{toke.carlsen@gmail.com}

\author{Aidan Sims}
\address[A. Sims]{School of Mathematics and Applied Statistics\\
University of Wollongong\\
NSW 2522\\
Australia}
\email{asims@uow.edu.au}

\date{\today}
\thanks{This research was supported by the Australian Research Council. Much of the work was completed between
breathtaking hikes in the Faroe Islands; Aidan thanks Toke for his generous hospitality.
Aidan also deeply thanks the organising committee of the 2015 Abel Symposium.}

\begin{document}

\begin{abstract}
In 2004, Hong and Szyma\'nski produced a complete description of the primitive-ideal
space of the $C^*$-algebra of a directed graph. This article details a slightly different
approach, in the simpler context of row-finite graphs with no sources, obtaining an
explicit description of the ideal lattice of a graph algebra.
\end{abstract}

\maketitle

\section{Introduction}

The purpose of this paper is to present a new exposition, in a somewhat simpler setting,
of Hong and Szyma\'nski's description of the primitive-ideal space of a graph
$C^*$-algebra. Their analysis \cite{HS2004} relates the primitive ideals of $C^*(E)$ to
the maximal tails $T$ of $E$---subsets of the vertex set satisfying three elementary
combinatorial conditions (see page~\pageref{pg:tail}). In previous work with Bates and
Raeburn, Hong and Szyma\'nski had already studied the primitive ideals of $C^*(E)$ that
are invariant for its gauge action. Specifically, \cite[Theorem~4.7]{BHRS} shows that the
gauge-invariant primitive ideals of $C^*(E)$ come in two flavours: those indexed by
maximal tails in which every cycle has an entrance; and those indexed by \emph{breaking
vertices}, which receive infinitely many edges in $E$, but only finitely many in the
maximal tail that they generate. Hong and Szyma\'nski completed this list by showing in
\cite[Theorem~2.10]{HS2004} that the non-gauge-invariant primitive ideals are indexed by
pairs consisting of a maximal tail containing a cycle with no entrance, and a complex
number of modulus~1.

The bulk of the work in \cite{HS2004} then went into the description of the Jacobson, or
hull-kernel, topology on $\Prim C^*(E)$ in terms of the indexing set described in the
preceding paragraph. Theorem~3.4 of \cite{HS2004} describes the closure of a subset of
$\Prim C^*(E)$ in terms of the combinatorial data of maximal tails and breaking vertices,
and the usual topology on the circle $\TT$. (Gabe \cite{Gabe} subsequently pointed out
and corrected a mistake in \cite[Theorem~3.4]{HS2004}, but there is no discrepancy for
row-finite graphs with no sources.) The technical details and notation involved even in
the statement of this theorem are formidable, with the upshot that applying Hong and
Szyma\'nski's result requires discussion of a fair amount of background and notation.
This is due to some extent to the complications introduced by infinite receivers in the
graph (to see this, compare \cite[Theorem~3.4]{HS2004} with the corresponding statement
\cite[Corollary~3.5]{HS2004} for row-finite graphs). But it is also caused in part by the
numerous cases involved in describing how the different flavours of primitive ideals
described in the preceding paragraph relate to one another topologically.

Here we restrict attention to the class of row-finite graphs with no sources originally
considered in \cite{KPRR, KPR, BPRS}; it is a well-known principal that results tend to
be cleaner in this context. The $C^*$-algebra of an arbitrary graph $E$ is a full corner
of the $C^*$-algebra of a row-finite graph $E_{ds}$ with no sources, called a
Drinen--Tomforde desingularisation $E$ \cite{DrinenTomforde}, so in principal our results
combined with the Rieffel correspondence can be used to describe the primitive-ideal
space and the ideal lattice of any graph $C^*$-algebra. But in practice there is serious
book-keeping hidden in this innocuous-sounding statement.

We take a somewhat different approach than Hong and Szyma\'nski. We start, as they do, by
identifying all the primitive ideals (Theorem~\ref{thm:prim ideal catalogue})---though we
take a slightly different route to the result. Our next step is to state precisely when a
given primitive ideal in our list belongs to the closure of some other set of primitive
ideals (Theorem~\ref{thm:1graph}). We could then describe the closure operation along the
lines of Hong and Szymanski's result, but here our approach diverges from theirs. We
describe a list of (not necessarily primitive) ideals $J_{H,U}$ of $C^*(E)$ indexed by
\emph{ideal pairs}, consisting of a saturated hereditary set $H$ and an assignment $U$ of
a proper open subset of the circle to every cycle with no entrance in the complement of
$H$. We describe each $J_{H,U}$ concretely by providing a family of generators. We prove
that the map $(H,U) \mapsto J_{H,U}$ is a bijection between ideal pairs and ideals, and
describe the inverse assignment (Theorem~\ref{thm:all ideals}). Finally, in
Theorem~\ref{thm:lattice}, we describe the containment relation and the intersection and
join operations on primitive ideals in terms of a partial ordering and a meet and a join
operation on ideal pairs.

One can recover the closure of a subset $X \subseteq \Prim C^*(E)$, and so Hong and
Szyma\'n\-ski's result, either by using the characterisation of points in $\overline{X}$
from Theorem~\ref{thm:1graph}, or by computing $\bigcap X$ using
Theorem~\ref{thm:lattice} and listing all the primitive ideals that contain this
intersection. To aid in doing the latter, we single out the ideal pairs that correspond
to primitive ideals (Remark~\ref{rmk:which primitive}), and identify when a given
$J_{H,U}$ is contained in a given primitive ideal (Lemma~\ref{lem:containment}).

We hope that this presentation of the ideal structure of $C^*(E)$ when $E$ is row-finite
with no sources will provide a useful and gentle introduction to Hong and Szyma\'nski's
beautiful result for arbitrary graphs; and in particular that it will be helpful to
readers familiar with the usual listing of gauge-invariant ideals using saturated
hereditary sets.

\textbf{Acknowledgement.} The exposition of this paper has benefitted greatly from the
suggestions of a very helpful referee. Thanks, whoever you are.

\subsection{Background}
We assume familiarity with Raeburn's monograph \cite{RaeburnCBMS} and take most of our
notation and conventions from there. We have made an effort not to assume any further
background.

We deal with row-finite directed graphs $E$ with no sources; these consist of countable
sets $E^0$, $E^1$ and maps $r,s : E^1 \to E^0$ such that $r$ is surjective and
finite-to-one. A Cuntz--Krieger family consists of projections $\{p_v : v \in E^0\}$ and
partial isometries $\{s_e : e \in E^1\}$ such that $s_e^*s_e = p_{s(e)}$ and $p_v =
\sum_{r(e) = v} s_e s^*_e$. We will use the convention where, for example, for $v \in
E^0$ the notation $vE^1$ means $\{e \in E^1 : r(e) = v\}$. A path of length $n > 0$ is a
string $\mu = e_1 \dots e_n$ of edges where $s(e_i) = r(e_{i+1})$, and $E^n$ denotes the
collection of paths of length $n$. We write $E^*$ for the collection of all finite paths
(including the vertices, regarded as paths of length~0), and set $vE^*:=\{\mu\in E^* :
r(\mu)=v\}$, $E^*w:=\{\mu\in E^* : s(\mu)=w\}$ and $vE^*w=vE^*\cap E^*w$ when $v,w\in
E^0$.

\section{Infinite paths and maximal tails}

Our first order of business is to relate maximal tails in a graph with the shift-tail
equivalence classes of infinite paths (see also \cite{IKSW}).

Recall that a \emph{maximal tail} in $E^0$ is a set $T \subseteq E^0$ such that:
\begin{itemize}[labelindent=3.3\parindent,leftmargin=*]\label{pg:tail}
    \item[(T1)] if $e \in E^1$ and $s(e) \in T$, then $r(e) \in T$;
    \item[(T2)] if $v \in T$ then there is at least one $e \in vE^1$ such that $s(e)
        \in T$; and
    \item[(T3)] if $v,w \in T$ then there exist $\mu \in vE^*$ and $\nu \in wE^*$
        such that $s(\mu) = s(\nu) \in T$.
\end{itemize}

If $T$ is a maximal tail, there is a subgraph $ET$ of $E$ with vertices $T$ and edges
$E^1 T:=\{e\in E^1 : s(e)\in T\}$.

An \emph{infinite path} in $E$ is a string $x = e_1 e_2 e_3 \cdots$ of edges such that
$s(e_i) = r(e_{i+1})$ for all $i$. We let $r(x):=r(e_1)$. Two infinite paths $x$ and $y$ are shift-tail
equivalent if there exist $m,n \in \NN$ such that
\[
    x_{i+m} = y_{i+n}\quad\text{ for all } i \in \NN.
\]
This shift-tail equivalence is (as the name suggests) an equivalence relation, and we
write $[x]$ for the equivalence class of an infinite path $x$.

Shift-tail equivalence classes $[x]$ of infinite paths correspond naturally to
irreducible representations of $C^*(E)$ (see Lemma~\ref{lem:pixz}). However, the
corresponding primitive ideals depend not on $[x]$, but only on the maximal tail
consisting of vertices that are the range of an infinite path in $[x]$. The next lemma
describes the relationship between shift-tail equivalence classes of infinite paths and
maximal tails.

\begin{lem}\label{lem:tails and orbits}
Let $E$ be a row-finite graph with no sources. A set $T \subseteq E^0$ is a maximal tail
if and only if there exists $x \in E^\infty$ such that $T = [x]^0 := \{r(y) : y \in
[x]\}$.
\end{lem}
\begin{proof}
First suppose that $T$ is a maximal tail. List $T = (v_1, v_2, \dots)$. Set $\lambda_1 =
\mu_1 = v_1 \in E^*$, and then inductively, having chosen $\mu_{i-1} \in v_{i-2} E^*$ and
$\lambda_{i-1} \in v_{i-1}E^*$ with $s(\lambda_{i-1}) = s(\mu)$, use~(T3) to find $\mu_i
\in v_{i-1}E^*$ and $\lambda_i \in v_i E^*$ such that $s(\mu_i) = s(\lambda_i) \in T$. We
obtain an infinite path $x = \mu_1\mu_2\mu_3 \cdots$. Since each
$\lambda_i\mu_{i+1}\mu_{i+2} \cdots$ belongs to $[x]$, we have $T \subseteq [x]^0$. For
the reverse containment, observe that if $v \in [x]^0$, then there exists $y \in [x]$
such that $v = r(y_1)$. By definition of $[x]$ there are $m,i$ such that $s(y_m) =
s(\mu_i)$. Since $\mu_i \in T$, $m$ applications of~(T1) show that $r(y_1) \in T$.
\end{proof}

We divide the maximal tails in $E$ into two sorts. Those which have a cycle with no
entrance, and those which don't. The main point is that, as pointed out in \cite{HS2004},
if $T$ contains a cycle without an entrance, then it contains just one of them, and is
completely determined by this cycle.

A \emph{cycle} in a graph $E$ is a path $\mu = \mu_1 \dots \mu_n\in E^*$ such that $r(\mu_1) =
s(\mu_n)$ and $s(\mu_i) \not= s(\mu_j)$ whenever $i \not= j$. Each cycle $\mu$
determines an infinite path $\mu^{\infty} := \mu\mu\mu\cdots$ and hence a maximal tail
$T_\mu := [\mu^\infty]^0$; it is straightforward to check that
\[
T_\mu = \{r(\lambda) : \lambda \in E^*r(\mu)\}.
\]

Given a cycle $\mu\in E^*$ and a subset $A$ of $E^0$ that contains $\{r(\mu_i) : i \le
|\mu|\}$, we say that \emph{$\mu$ is a cycle with no entrance in $A$} if $\{e\in r(\mu_i)E^1 : s(e)\in A\}=\{\mu_i\}$ for each $1 \le i \le |\mu|$.

\begin{lem}\label{lem:types of tails}
Let $E$ be a row-finite graph with no sources. Suppose that $T \subseteq E^0$ is a
maximal tail. Then either
\begin{itemize}[labelindent=\parindent,leftmargin=*]
    \item[a)] there is a cycle $\mu$ with no entrance in $T$ such that $T = T_\mu$, and
        this $\mu$ is unique up to cyclic permutation of its edges; or
    \item[b)] there is no cycle $\mu$ with no entrance in $T$.
\end{itemize}
\end{lem}
\begin{proof}
Suppose that there is a cycle $\mu$ with no entrance in $T$. Lemma~\ref{lem:tails and
orbits} implies that $T = [x]^0$ for some infinite path $x$. So there exists $y \in [x]$
such that $r(y) = r(\mu)$, and since shift-tail equivalence is an equivalence relation,
we then have $T = [y]^0$. Since $\mu$ has no entrance in $T$, the only element of
$E^\infty$ lying entirely within $T$ and with range $r(\mu)$ is $\mu^\infty$. So $y =
\mu^\infty$, and $T = [\mu^\infty]^0 = T_\mu$.

If $\nu$ is another cycle with no entrance in $T = T_\mu$ then $r(\nu) E^* r(\mu) \not=
\emptyset$, say $\lambda \in r(\nu) E^* r(\mu)$. Since $\nu$ has no entrance in $T$, we
have $\lambda\mu = \nu^\infty_1 \cdots \nu^\infty_k$ for some $k$. In particular
$\nu^\infty_{k-|\mu|+1} \cdots \nu^\infty_k = \mu$, and we deduce that $\nu = \mu_i
\cdots \mu_{|\mu|}\mu_1 \cdots \mu_{i-1}$, where $i \equiv
k+1~(\operatorname{mod}\;|\mu|)$.
\end{proof}

We call a maximal tail $T$ satisfying~(a) in Lemma~\ref{lem:types of tails} a
\emph{cyclic maximal tail} and write $\Per(T) := |\mu|$. We call a maximal tail $T$
satisfying~(b) in Lemma~\ref{lem:types of tails} a \emph{aperiodic maximal tail}, and
define $\Per(T) := 0$.

\section{The irreducible representations}

In this section, we show that every primitive ideal of $C^*(E)$ naturally determines a
corresponding maximal tail, and then construct a family of irreducible representations of
$C^*(E)$ associated to each maximal tail of $E$.

The following lemma constructs a maximal tail from each primitive ideal of $C^*(E)$. It
was proved for arbitrary graphs in \cite[Lemma~4.1]{BHRS} using the relationship between
ideals and saturated hereditary sets established there and that primitive ideals of
separable $C^*$-algebras are prime. Here we present instead the direct
representation-theoretic argument of \cite[Theorem~5.3]{CKSS}. Recall that a saturated
hereditary subset of $E^0$ is a subset whose complement satisfies axioms (T1)~and~(T2) of
a maximal tail.

\begin{lem}[{\cite[Lemma~4.1]{BHRS}}]\label{lem:prim->tail}
Let $E$ be a row-finite graph with no sources. If $I$ is a primitive ideal of $C^*(E)$,
then $T := \{v \in E^0 : p_v \not\in I\}$ is a maximal tail of $E$.
\end{lem}
\begin{proof}
The set of $v \in E^0$ such that $p_v \in I$ is a saturated hereditary set by
\cite[Lemma~4.5]{RaeburnCBMS} (see also \cite[Lemma~4.2]{BPRS}). So its complement $T$
satisfies (T1)~and~(T2). To establish~(T3), fix $v,w \in T$. Take an irreducible
representation $\pi : C^*(E) \to \Bb(\Hh)$ such that $\ker(\pi) = I$. Since $v
\in T$, we have $p_v \not\in I$, and so $\pi(p_v)\Hh \not= \{0\}$. Fix $\xi \in
\pi(p_v)\Hh$ with $\|\xi\| = 1$. Since $p_w \not\in I$, the space $\pi(p_w)\Hh$ is also a
nontrivial subspace of $\Hh$. Since $\pi$ is irreducible, $\xi$ is cyclic for $\pi$, and
so there exists $a \in C^*(E)$ such that $\pi(p_w)\pi(a)\xi = \pi(p_w a p_v)\xi$ is
nonzero. In particular, we have $\pi(p_w a p_v) \not= 0$. Since $C^*(E) = \clsp\{s_\mu
s^*_\nu : s(\mu) = s(\nu)\}$, and since $p_w s_\mu s^*_\nu p_v \not= 0$ only if $r(\mu) =
w$ and $r(\nu) = v$, we have
\[
\pi(p_w a p_v) \in \clsp\{\pi(s_\mu s^*_\nu) : r(\mu) = w, r(\nu) = v, s(\mu) = s(\nu)\} \setminus \{0\}.
\]
So there exist $\mu,\nu \in E$ with $r(\mu) = w$, $r(\nu) = v$, $s(\mu)= s(\nu)$, and
$\pi(s_\mu p_{s(\mu)} s^*_\nu) = \pi(s_\mu s^*_\nu) \not= 0$. In particular,
$\pi(p_{s(\mu)}) \not= 0$, giving $p_{s(\mu)} \not\in I$. So $s(\mu) \in T$ satisfies
$wE^* s(\mu), vE^*s(\mu) \not= \emptyset$.
\end{proof}

Next we show how to recover a family of primitive ideals from the shift-tail equivalence
class of an infinite path.

\begin{lem}\label{lem:pixz}
Let $E$ be a row-finite directed graph with no sources. For $x \in E^\infty$ and $z \in
\TT$, there is an irreducible representation $\pi_{x,z} : C^*(E) \to \Bb(\ell^2([x]))$
such that for all $y \in [x]$, $v \in E^0$ and $e \in E^1$, we have
\[
\pi_{x,z}(p_v) \delta_y = \begin{cases} \delta_y&\text{ if $r(y_1) = v$} \\ 0 &\text{ otherwise}\end{cases}
    \qquad\text{ and }\qquad
\pi_{x,z}(s_e) \delta_y = \begin{cases} z\delta_{ey}&\text{ if $r(y_1) = s(e)$} \\ 0 &\text{ otherwise.}\end{cases}
\]
We have $\{v \in E^0 : p_v \not\in \ker(\pi_{x,z})\} = [x]^0$.
\end{lem}
\begin{proof}
It is easy to check that $\ell^2([x])$ is an invariant subspace of $\ell^2(E^\infty)$ for
the infinite-path space representation of \cite[Example~10.2]{RaeburnCBMS} (with $k=1$).
So the infinite-path space representation reduces to a representation on
$\Bb(\ell^2([x]))$. Precomposing with the gauge automorphism $\gamma_z : s_e \mapsto z
s_e$ of \cite[Proposition~2.1]{RaeburnCBMS} yields a representation $\pi_{x,z}$
satisfying the desired formula.

To see that $\pi_{x,z}$ is irreducible, first observe that for each $x$, the rank-1
projection $\theta_{x,x}$ onto $\CC \delta_x$ is equal to the strong limit
\[
\theta_{x,x} = \lim_{n \to \infty} \pi_{x,z}(s_{x_1 \cdots x_n} s^*_{x_1 \cdots x_n}).
\]
If $y,z \in [x]$, then $y = \mu w$ and $z = \nu w$ for some $\mu,\nu \in E^*$ and $w \in
[x]$. Thus the rank-1 operator $\theta_{y,z}$ from $\CC \delta_z$ to $\CC \delta_y$ is in
the strong closure of the image of $\pi_{x,z}$:
\[
\theta_{y,z} = z^{|\nu| - |\mu|} \pi_{x,z}(s_\mu) \theta_{w,w} \pi_{x,z}(s^*_\nu)
    = \lim_{n \to \infty} \pi_{x,z}(z^{|\nu| - |\mu|} s_{\mu w_1 \cdots w_n} s^*_{\nu w_1 \cdots w_n}).
\]
So $\Kk(\ell^2([x]))$ is contained in the strong closure of $\pi_{x,z}(C^*(E))$. Thus
$\pi_{x,z}$ is irreducible.

If $v \not\in [x]^0$, then $v \not= r(y_1)$ for any $y \in [x]$, and so the formula for
$\pi_{x,z}$ shows that $p_v \in \ker(\pi_{x,z})$. On the other hand, if $v \in [x]^0$,
then we can find $y \in [x]$ with $r(y_1) = v$, and then $\pi_{x,z}(p_v) \delta_y =
\delta_y \not= 0$.
\end{proof}

Next we want to know when two of the irreducible representations constructed as in
Lemma~\ref{lem:pixz} have the same kernel. For the following, recall that if $H \subseteq
E^0$ is a hereditary set (i.e., $E^0\setminus H$ satisfies axiom (T1) of a
maximal tail.), then $E \setminus EH$ is the subgraph of $E$ with vertices $E^0
\setminus H$ and edges $E^1 \setminus E^1 H$. Note that if $T$ is a maximal tail, then $H
:= E^0 \setminus T$ is a saturated hereditary set, and then $E \setminus EH = ET$.

\begin{prp}\label{prp:same ideal}
Let $E$ be a row-finite graph with no sources. Fix $x,y \in E^\infty$ and $w,z \in \TT$.
The irreducible representations $\pi_{x,w}$ and $\pi_{y,z}$ have the same kernel if and
only if $[x]^0 = [y]^0$ and $w^{\Per([x]^0)} = z^{\Per([x]^0)}$.
\end{prp}

The crux of the proof of Proposition~\ref{prp:same ideal} is Lemma~\ref{lem:irreps},
which we state separately because it is needed again later to prove that every primitive
ideal is of the form $I_{\pi,z}$. Our proof of Lemma~\ref{lem:irreps} in turn relies on
the following standard fact about kernels of irreducible representations; we thank the
anonymous referee for suggesting the following elementary proof.

\begin{lem}\label{lem:irrep kernels}
Let $A$ be a $C^*$-algebra, let $J$ be an ideal of $A$, and let $\pi_1$ and $\pi_2$ be
irreducible representations of $A$ that do not vanish on $J$. Then $\ker(\pi_1) =
\ker(\pi_2)$ if and only if $\ker(\pi_1) \cap J = \ker(\pi_2) \cap J$.
\end{lem}
\begin{proof}
The ``$\implies$'' direction is obvious. Suppose that $\ker(\pi_1) \cap J = \ker(\pi_2)
\cap J$. By symmetry, it suffices to show that $\ker(\pi_1) \subseteq \ker(\pi_2)$. Since
$\pi_2$ is irreducible, $\ker(\pi_2)$ is primitive, and hence prime (see, for example,
\cite[Proposition~3.13.10]{Pedersen}). By assumption, we have $\ker(\pi_1) \cap J =
\ker(\pi_2) \cap J \subseteq \ker(\pi_2)$. Since $\pi_2$ does not vanish on $J$, we have
$J \not\subseteq \ker(\pi_2)$. So primeness of $\ker(\pi_2)$ forces $\ker(\pi_1)
\subseteq \ker(\pi_2)$.
\end{proof}

\begin{lem}\label{lem:irreps}
Let $E$ be a row-finite graph with no sources, and suppose that $T$ is a maximal tail of
$E$. Let $H := E^0 \setminus T$.
\begin{enumerate}[label=(\arabic*),labelindent=\parindent,leftmargin=*]
\item\label{it:irreps1} Suppose that $T$ is an aperiodic tail and $\pi$ is an
    irreducible representation of $C^*(E)$ such that $\{v \in E^0 : \pi(p_v) \not=
    0\} = T$. Then $\ker\pi$ is generated as an ideal by $\{p_v : v \in H\}$.
\item\label{it:irreps2} Suppose that $T$ is a cyclic tail and that $\mu$ is a cycle
    with no entrance in $T$. Suppose that $\pi_1$ and $\pi_2$ are irreducible
    representations of $C^*(E)$ such that
    \[
        \{v : \pi_1(p_v) \not= 0\} = T = \{v : \pi_2(p_v) \not= 0\}.
    \]
    Then each $\pi_i$ restricts to a 1-dimensional representation of $C^*(s_\mu)$,
    and $\ker\pi_1 = \ker\pi_2$ if and only if $\pi_1(s_\mu) = \pi_2(s_\mu)$ as
    complex numbers. Each $\ker\pi_i$ is generated as an ideal by $\{p_v : v \in H\}
    \cup \{\pi_i(s_\mu) p_{r(\mu)} - s_\mu\}$.
\end{enumerate}
\end{lem}
\begin{proof}
We start with some setup that is needed for both statements. Let $I$ be the ideal of
$C^*(E)$ generated by $\{p_v : v \in H\}$. This $H$ is a saturated hereditary set. If
$\pi$ is an irreducible representation such that $\{v \in E^0 : \pi(p_v) \not= 0\} = T$,
then $I$ is contained in $\ker\pi$ by definition. By \cite[Remark~4.12]{RaeburnCBMS},
there is an isomorphism $C^*(E)/I \cong C^*(E \setminus EH)$ that carries $p_v + I$ to
$p_v$ for $v \in E^0 \setminus H$. Since $I \subseteq \ker\pi$, the representation $\pi$
descends to an irreducible representation of $C^*(E)/I$, and hence determines a
representation $\tilde\pi$ of $C^*(E \setminus EH)$ such that
\[
    \tilde\pi(p_v) = \pi(p_v)\quad\text{ for $v \in E^0 \setminus H$.}
\]

Now, for~\ref{it:irreps1}, if $T$ is an aperiodic maximal tail, and $\pi$ is as above,
then every cycle in $E\setminus H$ has an entrance in $E\setminus H$, and $\tilde\pi$ is a representation
of $C^*(E \setminus EH)$ such that $\tilde\pi(p_v) \not= 0$ for all $v \in (E \setminus
EH)^0$. So the Cuntz--Krieger uniqueness theorem \cite[Theorem~2.4]{RaeburnCBMS} implies
that $\tilde\pi$ is faithful. Hence $\ker\pi = I$, proving~(1).

For~\ref{it:irreps2}, consider the ideal $J$ of $C^*(E \setminus EH)$ generated by
$p_{r(\mu)}$. Then $\tilde\pi_i(J) \not= \{0\}$ for $i = 1,2$. So Lemma~\ref{lem:irrep
kernels} implies that $\tilde\pi_1$ and $\tilde\pi_2$ have the same kernel if and only if
$\ker(\tilde\pi_1) \cap J = \ker(\tilde\pi_2) \cap J$. Since $J$ is generated as an ideal
by $p_{r(\mu)}$, the corner $p_{r(\mu)} J p_{r(\mu)} = \clsp\{s_{\mu^n} s^*_{\mu^m} : m,
n \in \NN\}$ is full in $J$. Rieffel induction from a $C^*$-algebra to a full corner is
implemented by restriction of representations \cite[Proposition~3.24]{tfb}. Since Rieffel
induction carries irreducible representations to irreducible representations and induces
a bijection between primitive-ideal spaces, we deduce that each $\tilde\pi_i$ is an
irreducible representation of $C^*(s_\mu) \subseteq J$, and that
\[
\ker\tilde\pi_1 = \ker\tilde\pi_2 \qquad\Longleftrightarrow\qquad
    \ker(\tilde\pi_1) \cap p_{r(\mu)}Jp_{r(\mu)} = \ker(\tilde\pi_2) \cap p_{r(\mu)}Jp_{r(\mu)}.
\]
Since $\mu$ has no entrance, $s_\mu$ is a unitary element of $p_{r(\mu)} J p_{r(\mu)}$,
so $C^*(s_\mu) \cong C(\sigma(s_\mu))$. Since the irreducible representations of a
commutative $C^*$-algebra are 1-dimensional, we deduce that each $\tilde\pi_i$ is a
1-dimensional representation of $C^*(s_\mu) \subseteq C^*(E \setminus EH)$ and hence each
$\pi_i$ is a 1-dimensional representation of $C^*(s_\mu) \subseteq C^*(E)$. Moreover,
$\tilde\pi_1$ and $\tilde\pi_2$ have the same kernel if and only if they are implemented
by evaluation at the same point $z$ in $\sigma(s_\mu)$, and hence if and only if
$\pi_1(s_\mu) = \pi_2(s_\mu)$.

For the final statement fix $i \in \{1,2\}$. Since $I$ is contained in the ideal $J'$ generated by
$\{p_v : v \in H\} \cup \{\pi_i(s_\mu)p_{r(\mu)} - s_\mu\}$, we have $J' = \ker\pi_i$ if and only
if $\ker\tilde\pi_i$ is equal to the image $J''$ of $J'/I$ in $C^*(E \setminus EH)$. Let $z :=
\pi_i(s_\mu) \in \TT$. By definition we have $\pi_i(s_\mu - zp_{r(\mu)}) = z - z = 0$, and so $J''
\subseteq \ker(\tilde\pi_i)$. We must establish the reverse inclusion.

Write $\mu = e_1 e_2 \cdots e_m$ with each $e_m \in E^1$, and let $F$ be the directed graph with
$F^0 = E^0 \setminus H$ and $F^1 = (E^1 \setminus E^1 H) \setminus \{e_m\}$, and with range and
source maps inherited from $E$. For $v \in F^0$, define $q_v := p_v + J'' \in C^*(E \setminus
EH)/J''$ and for $f \in F^1$ define $t_f = s_f + J'' \in C^*(E \setminus EH)/J''$. Then $\{q_v,
t_f\}$ is a Cuntz--Krieger $F$-family in $C^*(E \setminus EH)/J''$. So there is a homomorphism
$\phi : C^*(F) \to C^*(E \setminus EH)/J''$ such that $\phi(p_v) = q_v$ and $\phi(s_f) = t_f$.
Since $zp_{r(\mu)}-s_\mu\in J''$,  we have
\[
s_{e_m}^* + J'' = p_{r(\mu)}s_{e_m}^* + J'' = \overline{z} s_{e_1} \dots s_{e_{m-1}} s_{e_m} s^*_{e_m} + J''.
\]
Since $\mu$ has no entrance in $E \setminus EH$, we have $s_{e_m} s^*_{e_m} = p_{r(e_m)}$. Thus
$s_{e_m} + J'' = z s_{e_1 \cdots e_{m-1}}^* + J''$ is in the range of $\phi$, and we deduce that
$\phi$ is surjective.

Since $J'' \subseteq \ker\tilde\pi_i$, the homomorphism $\tilde\pi_i$ descends to a homomorphism
$\bar\pi_i : C^*(E \setminus EH)/J'' \to \tilde\pi_i(C^*(E \setminus EH))$. We claim that $\bar\pi
\circ \phi$ is injective. To see this, observe that since $\pi_i$ is 1-dimensional on $C^*(\mu)$,
it is nonzero at $p_{r(\mu)}$, and then since $E^0$ is a maximal tail, we see that $\pi_i(p_v)
\not= 0$ for all $v \in E^0 \setminus H$. Hence $\bar\pi_i \circ \phi(p_v) \not= 0$ for all $v \in
F^0$. Since every cycle in $F$ has an entrance, the Cuntz--Krieger uniqueness theorem now implies
that $\bar\pi \circ \phi$ is injective. It follows that $\bar\pi$ is injective on the image of
$\phi$. We established above that $\phi$ is surjective, and so we deduce that $\bar\pi$ is
injective on $C^*(E \setminus EH)/J''$, which says precisely that $\ker\tilde{\pi}_i \subseteq
J''$.
\end{proof}

\begin{rmk*}
The final three paragraphs of the proof of Lemma~\ref{lem:irreps} above replace the final
paragraph of the same proof in the published version \cite{CS} of this paper, and fix a gap in the
proof which was brought to our attention by Hui Li. In the final paragraph of the proof in
\cite{CS}, we claimed that to see that $\ker(\tilde\pi_i) = J''$, it suffices to show that
$p_{r(\mu)}{\ker(\tilde\pi_i)}p_{r(\mu)} = p_{r(\mu)} J'' p_{r(\mu)}$. The error is that this
shows only that $\ker(\tilde\pi_i) \cap J = J'' \cap J$. To deduce the desired equality
$\ker(\tilde\pi_i) = J''$, we had intended that this followed from an appeal to
\cite[Lemma~3.4]{CS}, but for this, we would need to know that $J''$ is a primitive ideal. We
thank Hui for pointing out the error to us.
\end{rmk*}

\begin{proof}[Proof of Proposition~\ref{prp:same ideal}]
The final statement of Lemma~\ref{lem:pixz} implies that if $\ker\pi_{x,w} =
\ker\pi_{y,z}$, then $[x]^0 = [y]^0$. So it suffices to prove that if $[x]^0 = [y]^0$,
then
\begin{equation}\label{eq:kers vs chars}
\ker\pi_{x,w} = \ker\pi_{y,z}\text{ if and only if }w^{\Per([x]^0)} = z^{\Per([x]^0)}.
\end{equation}

For this we consider two cases. First suppose that $[x]^0$ is an aperiodic maximal tail.
Then Lemma~\ref{lem:irreps}\ref{it:irreps1} implies that each of $\ker\pi_{x,w}$ and
$\ker\pi_{y,z}$ is generated by $\{p_v : v \not\in T\}$, and in particular the two are
equal. Also, $w^{\Per([x]^0)} = w^0 = 1 = z^0 = z^{\Per([x]^0)}$, so the
equivalence~\eqref{eq:kers vs chars} holds.

Now suppose that $[x]^0$ is cyclic, and let $\mu$ be a cycle with no entrance in $[x]^0$.
We must show that $\ker\tilde\pi_{x,w} = \ker\tilde\pi_{y,z}$ if and only if $w^{|\mu|} =
z^{|\mu|}$. Since $\mu$ has no entrance, both $\pi_{x,w}(p_{r(\mu)})\ell^2([x])$ and
$\pi_{y,z}(p_{r(\mu)})\ell^2([y])$ are equal to the 1-dimensional space $\CC
\delta_{\mu^\infty}$, and we have
\[
\pi_{x,w}(s_\mu) \delta_{\mu^\infty}
    = w^{|\mu|} \delta_{\mu^\infty}\qquad\text{ and }\qquad
\pi_{y,z}(s_\mu) \delta_{\mu^\infty}
    = z^{|\mu|}\delta_{\mu^\infty}.
\]
So, identifying the image of $\pi_{x,w}(C^*(s_\mu))$ with $\CC$, we have
$\pi_{x,w}(s_\mu) = w^{|\mu|}$ and similarly $\pi_{y,z}(s_\mu) = z^{|\mu|}$. So
Lemma~\ref{lem:irreps}\ref{it:irreps2} shows that $\ker\pi_{x,w} = \ker\pi_{y,z}$ if
and only if $z^{|\mu|} = w^{|\mu|}$.
\end{proof}

We are now ready to state and prove our first main result---a catalogue of the primitive
ideals of $C^*(E)$. Proposition~\ref{prp:same ideal} says that the following definition
makes sense.

\begin{dfn}
Let $E$ be a row-finite directed graph with no sources. Suppose that $T$ is a maximal
tail in $E^0$ and that $z \in \{w^{\Per(T)} : w \in \TT\} \subseteq \TT$. We define
\[
I_{T,z} := \ker\pi_{x,w} \text{ for any $(x,w) \in E^\infty \times \TT$ such that $[x]^0 = T$ and $w^{\Per(T)} = z$.}
\]
\end{dfn}

\begin{thm}\label{thm:prim ideal catalogue}
The map $(T,z) \mapsto I_{T,z}$ is a bijection from
\[
    \{(T, w^{\Per(T)}) : T\text{ is a maximal tail, }w \in \TT\}
\]
to $\Prim C^*(E)$.
\end{thm}
\begin{proof}
Lemma~\ref{lem:pixz} shows that each $I_{T,z}$ is a primitive ideal.
Proposition~\ref{prp:same ideal} shows that $(T, z) \mapsto I_{T,z}$ is injective. So we
just have to show that it is surjective. Fix a primitive ideal $J$ of $C^*(E)$, let $T =
\{v : p_v \not\in J\}$, and let $\pi$ be an irreducible representation of $C^*(E)$ with
kernel $J$. Then $T$ is a maximal tail according to Lemma \ref{lem:prim->tail}.
We must show that $J$ has the form $I_{T, z}$.

If $T$ is aperiodic, then Lemma~\ref{lem:irreps}\ref{it:irreps1} shows that $J =
\ker\pi = \ker\pi_{x, 1} = I_{[x]^0, 1}$ for any $x$ such that $[x]^0 = T$.

If $T$ is cyclic, let $\mu$ be a cycle with no entrance in $T$.
Lemma~\ref{lem:irreps}\ref{it:irreps2} shows that $\pi(C^*(s_\mu))$ is one-dimensional,
so we can identify $\pi(s_\mu)$ with a nonzero complex number $z$. Since $s_\mu$ is an
isometry, $|z| = 1$. Now Lemma~\ref{lem:irreps}\ref{it:irreps2} implies that any $w \in
\TT$ with $w^{|\mu|} = z$ satisfies $\ker\pi = \ker\pi_{[\mu^\infty]^0, w} = I_{[x]^0,
z}$.
\end{proof}

\section{The closure operation}

The Jacobson, or hull-kernel, topology on $\Prim C^*(E)$ is the one determined by the
closure operation $\overline{X} = \{I \in \Prim C^*(E) : \bigcap_{J \in X} J
\subseteq I\}$. The ideals of $C^*(E)$ are in bijection with the closed subsets of $\Prim
C^*(E)$: the ideal $I_X$ corresponding to a closed subset $X$ is
\[\textstyle
I_X := \bigcap_{J \in X} J.
\]

So the first step in describing the ideals of $C^*(E)$ is to say when a primitive ideal
$I$ belongs to the closure of a set $X$ of primitive ideals. We do so with the following
theorem.

\begin{thm} \label{thm:1graph}
Let $E$ be a row-finite graph with no sources. Let $X$ be a set of pairs $(T, z)$
consisting of a maximal tail $T$ and an element $z$ of $\{w^{\Per(T)} : w \in \TT\}$.
Consider another such pair $(S, w)$. Then $\bigcap_{(T, z) \in X} I_{T, z} \subseteq
I_{S,w}$ if and only if both of the following hold:
\begin{enumerate}[label=\alph*),labelindent=\parindent,leftmargin=*]
\item\label{en:a} $S \subseteq \bigcup_{(T, z) \in X} T$, and
\item\label{en:b} if $S$ is a cyclic tail and the cycle $\mu$ with no entrance in $S$
    also has no entrance in $\bigcup_{(T,z) \in X} T$, then
    \[
        w \in \overline{\{z : (S, z) \in X\}}.
    \]
\end{enumerate}
\end{thm}

We will need the following simple lemma in the proof Theorem~\ref{thm:1graph}, and at a
number of other points later in the paper.

\begin{lem}\label{lem:corner}
Let $E$ be a row-finite graph with no sources, let $H$ be a saturated hereditary subset
of $C^*(E)$ and let $\mu$ be a cycle with no entrance in $E^0 \setminus H$. Let
$I_H$ be the ideal of $C^*(E)$ generated by $\{p_v : v \in H\}$. Then there is an
isomorphism
\[
(p_{r(\mu)} C^*(E) p_{r(\mu)}) / (p_{r(\mu)} I_H p_{r(\mu)}) \cong p_{r(\mu)} C^*(E \setminus EH) p_{r(\mu)}
\]
carrying $s_\mu + p_{r(\mu)} I_H p_{r(\mu)}$ to $s_\mu$, and there is an isomorphism of
$p_{r(\mu)} C^*(E \setminus EH) p_{r(\mu)}$ onto $C(\TT)$ carrying $s_\mu$ to the
generating monomial function $z \mapsto z$.
\end{lem}
\begin{proof}
Remark~4.12 of \cite{RaeburnCBMS} shows that there is an isomorphism $C^*(E)/I_H \cong
C^*(E \setminus EH)$ that carries $s_e + I_H$ to $s_e$ if $e \in E^1 \setminus
E^1H$ and to zero otherwise. This restricts to the desired isomorphism $p_{r(\mu)}
C^*(E) p_{r(\mu)}/ p_{r(\mu)} I_H p_{r(\mu)} \cong p_{r(\mu)} C^*(E \setminus EH)
p_{r(\mu)}$. The element $s_\mu \in p_{r(\mu)} C^*(E \setminus EH) p_{r(\mu)}$ satisfies
$s_\mu^* s_\mu  = p_{r(\mu)} = s_\mu s^*_\mu$ because $\mu$ has no entrance in
$E^0 \setminus H$. So it suffices to show that the spectrum of $s_\mu$
calculated in $p_{r(\mu)} C^*(E \setminus EH) p_{r(\mu)}$ is $\TT$. To see this, observe
that the gauge action $\gamma$ satisfies $\gamma_w(s_\mu) = w^{|\mu|}(s_\mu)$. So for
$\lambda, w \in \TT$, $\lambda p_{r(\mu)} - s_\mu$ is invertible if and only if
$\gamma_w(\lambda p_{r(\mu)} - s_\mu) = w^{|\mu|}(w^{-|\mu|}\lambda p_{r(\mu)} - s_\mu)$.
That is, $\sigma(\mu)$ is invariant under rotation by elements of the form $w^{|\mu|}$,
which is all of $\TT$. Since the spectrum is nonempty, it follows that it is the whole
circle.
\end{proof}

\begin{proof}[Proof of Theorem~\ref{thm:1graph}]
We first prove the ``if" direction. So suppose that~(\ref{en:a} and~(\ref{en:b} are
satisfied. We consider two cases. First suppose that $S$ is an aperiodic tail. Then
$\Per(S) = \{0\}$, and so $w = 1$. For each maximal tail $T$ of $E$, let
\[
T_- := T \setminus \{v : v\text{ lies on a cycle with no entrance in }T\},
\]
and let $I_{T_-}$ be the ideal generated by $\{p_v : v \not\in T_-\}$.
If $T$ is a cyclic maximal tail and $\mu$ is a cycle with no entrance in $T$, and if $z
\in \{w^{\Per(T)} : w \in \TT\}$, then Lemma~\ref{lem:irreps}\ref{it:irreps2} shows
that $I_{T, z}$ is generated by $\{p_v : v \not\in T\} \cup \{zp_{r(\mu)} - s_\mu\}$. So
$I_{T, z} \subseteq I_{T_-}$. So it suffices to show that
\[\textstyle
\bigcap_{(T, z) \in X} I_{T_-} \subseteq I_{S,1}.
\]
For this it suffices to show that $\bigcup_{(T, z) \in X} T_- \supseteq S$. We fix $v \in
E^0 \setminus \bigcup_{(T, z) \in X} T_-$ and show that $v \not\in S$. If $v \not\in T$
for all $(T, z) \in X$, then it follows from~(\ref{en:a} that $v \not\in S$. So we may
assume that $v \in \big(\bigcup_{(T, z) \in X} T\big) \setminus \big(\bigcup_{(T, Z) \in
X} T_-\big)$. In particular, there exist pairs $(T, z) \in X$ such that $v \in T$. Fix
any such pair. Since $v \not\in T_-$, it must lie in a cycle $\mu$ in $T$ with no
entrance in $T$. Property~(T1) shows that $\mu$ is contained entirely in $T$, and then
Lemma~\ref{lem:types of tails} then gives $T = [\mu^\infty]^0 = r(E^* v)$. So $\mu$ has
no entrance in $r(E^* v)$, and the only pairs $(T, z) \in X$ with $v \in T$ satisfy $T =
r(E^* v)$. Thus $\mu$ has no entrance in $\bigcup_{(T, z) \in X} T$. Since $S \subseteq
\bigcup_{(T, z) \in X} T$, and every cycle in $S$ has an entrance in $S$, we deduce that
$\mu$ does not lie in $S$ and hence $v \not\in S$ as required.

Now suppose that $S$ is cyclic and $\mu$ is a cycle with no entrance in $S$. Let $V$ be
the set of vertices on $\mu$. Lemma~\ref{lem:types of tails} gives $S = \{r(\alpha) :
s(\alpha) \in V\}$. Since $S \subseteq \bigcup_{(T, z) \in X} T$, there exists $(T, z)
\in X$ with $r(\mu) \in T$. Since $T$ satisfies~(T1), we deduce that the cycle $\mu$ lies
in the subgraph $ET$ of $E$. So there exists $(T, z) \in X$ such that $V
\subseteq T$, and then $S \subseteq T$ because $S = \{r(\alpha) : s(\alpha) \in V\}$ and
$T$ satisfies~(T1). So it suffices to show that
\[\textstyle
\bigcap_{(T, z) \in X, S \subseteq T} I_{T, z} \subseteq I_{S, \omega}.
\]
For this, first suppose that there exists $(T,z) \in X$ such that $T$ is a proper
superset of $S$; say $v \in T \setminus S$. Since $S = \{r(\alpha) : s(\alpha) \in V\}$,
we see that $v E^* V = \emptyset$, and hence $v E^* S = \emptyset$. So there exists $w
\in T \setminus S$ such that $V E^* w$ and $v E^* w$ are both nonempty. Hence
\[
T \supseteq \{r(\alpha) : s(\alpha) = w\} \supseteq \{r(\alpha) : s(\alpha) \in V\} = S.
\]
If $T$ is a cyclic tail, the cycle with no entrance that it contains lies outside of $S$,
so the final statement of Lemma~\ref{lem:irreps}\ref{it:irreps2} shows that all the
generators of $I_{T, z}$ belong to $I_{S, w}$; and if $T$ is aperiodic, then all the
generators of $I_{T, z}$ belong to $I_{S, w}$ by Lemma~\ref{lem:irreps}\ref{it:irreps1}.
In either case, we conclude that $I_{T, z} \subseteq I_{S, w}$, and hence $\bigcap_{(T,
z) \in X, S \subseteq T} I_{T, z} \subseteq I_{S, \omega}$.

So it now suffices to show that $\bigcap_{z : (S, z) \in X} I_{S, z} \subseteq I_{S, w}$.
Let $I_S$ be the ideal generated by $\{p_v : v \not\in S\}$. Then each $I_{S,z}$ contains
$I_S$, as does $I_{S,w}$, so we need only show that in the quotient $C^*(E)/I_S \cong
C^*(E S)$, the intersection of the images $J_z$ of the $I_{S, z}$ is contained in $J_w$.
Each $J_z$ is generated by $zp_{r(\mu)} - s_\mu$ and is therefore contained in the ideal
generated by $p_{r(\mu)}$, and similarly for $J_w$. Since the ideal generated by
$p_{r(\mu)}$ is Morita equivalent to the corner determined by $p_{r(\mu)}$, it suffices
to show that $\bigcap_{(S, z) \in X} p_{r(\mu)} J_z p_{r(\mu)} \subseteq p_{r(\mu)} J_w
p_{r(\mu)}$. The isomorphism $p_{r(\mu)} C^*(E S) p_{r(\mu)} \cong C(\TT)$ of
Lemma~\ref{lem:corner} carries each $p_{r(\mu)} J_z p_{r(\mu)}$ to $\{f \in C(\TT) : f(z)
= 0\}$. So $\bigcap_{(S, z) \in X} p_{r(\mu)} J_z p_{r(\mu)}$ is carried to $\big\{f \in
C(\TT) : f \equiv 0\text{ on } \overline{\{z : (S, z) \in X\}}\big\}$, and in particular
is contained in the image of $p_{r(\mu)} J_w p_{r(\mu)}$.

We now prove the ``only if" direction. To do this, we prove the contrapositive. So we
first suppose that~(\ref{en:a} does not hold. Then there is some $v \in S \setminus
\bigcup_{(T, z)} T$. This implies that $p_v \in I_{(T, z)}$ for all $(T, z)$, but $p_v
\not\in I_{S,w}$, and so $\bigcap_{(T, z)} I_{T, zi} \not\subseteq I_{S, w}$ as required.

Now suppose that $S \subseteq \bigcup_{(T, z)} T$, that $\mu$ is a cycle with no
entrance in $S$ and that $\mu$ also has no entrance in $\bigcup_{(T,z) \in X} T$, and
that $w \not\in \overline{\{z : (S, z) \in X\}}$. As above, $S = \{r(\alpha) : s(\alpha)
= r(\mu)\}$, and since $\mu$ has no entrance in any $T$, for each $(T, z)$ we have either
$T = S$ or $r(\mu) \not\in T$. Whenever $r(\mu) \not\in T$, we have $p_{r(\mu)} \in
I_{(T, z)}$, and so $\bigcap_{(T, z)} p_{r(\mu)} I_{T, z} p_{r(\mu)} = \bigcap_{(S, z)
\in X} p_{r(\mu)} I_{S, z} p_{r(\mu)}$. Once again taking quotients by $I_S$, it suffices
to show that
$$\bigcap_{(S, z) \in X} p_{r(\mu)}J_z p_{r(\mu)} \not\subseteq p_{r(\mu)}J_w
p_{r(\mu)}.$$ Since $w \not\in \overline{\{z : (S, z) \in X\}}$, there exists $f \in
C(\TT)$ such that $f(w) = 0$ and $f(z) = 1$ whenever $(S, z) \in X$. Let $g = 1-f \in
C(\TT)$. Then the images of the elements $f$ and $g$ belong to $\bigcap_{(S, z) \in X}
p_{r(\mu)} J_z p_{r(\mu)}$ and $p_{r(\mu)} J_w p_{r(\mu)}$ respectively. Their sum is the
identity element $p_{r(\mu)}$, which does not belong to $J_w$. Thus
\[
p_{r(\mu)} J_w p_{r(\mu)} + \bigcap_{(S, z) \in X} p_{r(\mu)} J_z p_{r(\mu)} \not= J_w.
\]
Consequently, $\bigcap_{(S, z) \in X} p_{r(\mu)}J_z p_{r(\mu)} \not\subseteq
p_{r(\mu)}J_w p_{r(\mu)}$.
\end{proof}

\section{The ideals of $C^*(E)$}

We use Theorem~\ref{thm:1graph} above to describe all the ideals of $C^*(E)$. We index
them by what we call ideal pairs for $E$. To define these, given a saturated hereditary
set $H$ of $E^0$, we will write $\CH$ for the set
\[
\CH := \{\mu : \mu\text{ is a cycle with no entrance in } E^0 \setminus H\}.
\]
An \emph{ideal pair} for $E$ is then a pair $(H, U)$ where $H$ is a saturated hereditary
set, and $U$ is a function assigning to each $\mu \in \CH$ a proper open subset $U(\mu)$
of $\TT$, with the property that $U(\mu) = U(\nu)$ whenever $[\mu^\infty] =
[\nu^\infty]$.

Observe that if the maximal tail $E^0 \setminus H$ is aperiodic, so that $\CH =
\emptyset$, then there is exactly one ideal pair of the form $(H, U)$: the function $U$
is the unique (trivial) function from the empty set to the collection of proper open
subsets of $\TT$.

To see how to obtain an ideal of $C^*(E)$ from an ideal pair, we need to do a little bit
of background work.

For each open subset $U \subseteq \TT$, we fix a function $h_U \in C(\TT)$ such that
\[
\{z \in \TT : h_U(z) \not= 0\} = U.
\]
For example, we could take
\[
h_U(z) := \inf\{|z - w| : w\notin U\}.
\]

Let $\pi : C(\TT) \to \ell^2(\ZZ)$ be the faithful representation that carries the
generating monomial $z \mapsto z$ to the bilateral shift operator $U : e_n \mapsto
e_{n+1}$. The classical theory of Toeplitz operators says that if $P_+ : \ell^2(\ZZ) \to
\ell^2(\NN)$ denotes the orthogonal projection onto the Hardy space $\clsp\{e_n : n \ge
0\}$, then there is an isomorphism $\rho$ from $P_+ \pi(C(\TT)) P_+$ to the Toeplitz
algebra $\Tt \subseteq \ell^2(\NN)$ generated by the unilateral shift operator $S$, such
that if $q : \Tt \to C(\TT)$ is the quotient map that divides out the ideal of compact
operators, then $q(\rho(P_+ \pi(f) P_+)) = f$ for every $f \in C(\TT)$.

If $H \subseteq E^0$ is saturated and hereditary, then for each $\mu \in \CH$, we have
$s_\mu s^*_\mu \le p_{r(\mu)} = s^*_\mu s_\mu$, with equality precisely if $\mu$ has
no entrance in $E^0$. So if $\mu$ has no entrance in $E^0$, then
$s_\mu$ is unitary in $p_{r(\mu)} C^*(E) p_{r(\mu)}$, and we can apply the functional
calculus in the corner to define a nonzero element $h_U(s_\mu) \in C^*(E)$. If $\mu$
has an entrance in $E^0$, then $s_\mu s^*_\mu < s^*_\mu s_\mu$, so Coburn's
theorem \cite{Coburn} gives an isomorphism $\psi : \Tt \cong C^*(s_\mu)$ that carries $S$
to $s_\mu$.

Using the preceding paragraph, given an ideal pair $(H,U)$ and given $\mu \in \CH$, we
obtain an element $\tau^U_\mu \in C^*(s_\mu) \subseteq p_{r(\mu)} C^*(E) p_{r(\mu)}$
given by
\begin{equation*}
\tau^U_\mu :=
    \begin{cases}
        h_{U(\mu)}(s_\mu) &\text{ if $\mu$ has no entrance in $E^0$} \\
        \psi(\rho(P_+ \pi(h_{U(\mu)}) P_+)) &\text{ otherwise.}
    \end{cases}
\end{equation*}

\begin{thm}\label{thm:all ideals}
Let $E$ be a row-finite graph with no sources. Let $\Ii_E$ denote the set of all ideal
pairs for $E$. For each $(H,U) \in \Ii_E$, let $J_{H,U}$ be the ideal of $C^*(E)$
generated by
\[
\{p_v : v \in H\} \cup \{\tau^U_\mu : \mu \in \CH\}.
\]
\begin{enumerate}[label=(\arabic*),labelindent=\parindent,leftmargin=*]
\item\label{it:pair->ideal} The map $(H,U) \mapsto J_{H,U}$ is a bijection of $\Ii_E$
    onto the collection of all closed 2-sided ideals of $C^*(E)$.
\item\label{it:ideal->pair} Given an ideal $I$ of $C^*(E)$, let $H_I := \{v \in E^0 :
    p_v \in I\}$, and for $\mu \in \Cc(H_I)$, let $U_I(\mu) = \TT \setminus
    \spec_{(p_{r(\mu)} + I) (C^*(E)/I) (p_{r(\mu)} + I)} (s_\mu + I)$. Then
    $(H_I,U_I)$ is an ideal pair and $I = J_{H_I,U_I}$.
\end{enumerate}
\end{thm}

Before proving the theorem, we need the following lemma.

\begin{lem}\label{lem:containment}
Let $E$ be a row-finite directed graph with no sources. Let $(H, U)$ be an ideal pair for
$E$, let $T$ be a maximal tail of $E$ and take $z \in \{w^{\Per(T)} : w \in \TT\}$. Then
$J_{H,U} \subseteq I_{T, z}$ if and only if both of the following hold:
\begin{enumerate}[label=\alph*),labelindent=\parindent,leftmargin=*]
	\item $H \subseteq E^0 \setminus T$; and
	\item if $T$ is cyclic and the cycle $\mu$ with no entrance in $T$
belongs to $\CH$, then $z \not\in U(\mu)$.
\end{enumerate}
In particular, we have $\{v : p_v \in
J_{H,U}\} = H$.
\end{lem}
\begin{proof}
For the ``if" direction, fix $x \in E^\infty$ such that $T = [x]^0$ and $w \in \TT$ such
that $w^{\Per(T)} = z$. We just have to show that $\pi_{x,w}$ annihilates all the
generators of $J_{H,U}$. For this, first fix $v \in H$. Then the final statement of
Lemma~\ref{lem:pixz} shows that $p_v \in \ker\pi_{x,w}$. Now fix $\mu \in \CH$. If
$r(\mu) \not\in T$, then $\pi_{x,w}(p_{r(\mu)}) = 0$ as above and then since $\tau^U_\mu
\in p_{r(\mu)} C^*(E) p_{r(\mu)}$, it follows that $\pi_{x, w}(\tau^U_\mu) = 0$. So
suppose that $r(\mu) \in T$. Since $\mu$ has no entrance in $E^0 \setminus H$ and since
$T \subseteq E^0 \setminus H$, the cycle $\mu$ has no entrance in $T$. So $T$ is a cyclic
maximal tail, and $[x]^0 = [\mu^\infty]^0$ by Lemma~\ref{lem:types of tails}. We then
have $z \not\in U(\mu)$ by hypothesis. The ideal $I_H$ generated by $\{p_v : v \in H\}$
is contained in $\ker(\pi_{x,w})$, so $\pi_{x,w}$ descends to a representation
$\tilde\pi_{x,w}$ of $C^*(E)/I_H$. Lemma~\ref{lem:corner} shows that $p_{r(\mu)} C^*(E)
p_{r(\mu)} / p_{r(\mu)} I p_{r(\mu)} \cong C(\TT)$, and this isomorphism carries the
restriction of $\tilde\pi_{x,w}$ to the 1-dimensional representation $\epsilon_z$ given
by evaluation at $z$. The isomorphism of Lemma~\ref{lem:corner} also carries $\tau^U_\mu
+ p_{r(\mu)} I p_{r(\mu)}$ to $h_{U(\mu)}$. Since $z \not\in U(\mu)$, we have
$\epsilon_z(h_{U(\mu)}) = 0$, and so $\pi_{x,w}(\tau^U_\mu) = 0$. So all of the
generators of $J_{H,U}$ belong to $\ker\pi_{x,w}$ as required.

For the ``only if" implication, we prove the contrapositive. Again fix $x \in
E^\infty$ such that $T = [x]^0$ and $w \in \TT$ such that $w^{\Per(T)} = z$, so that
$I_{T,z} = \ker\pi_{x,w}$. First suppose that $H \not\subseteq E^0 \setminus T$; say $v
\in T \cap H$. Then $p_v \in J_{H,U}$ by definition, but $p_v \not\in \ker\pi_{x,w}$ by
the final statement of Lemma~\ref{lem:pixz}, giving $J_{H,U} \not\subseteq
\ker\pi_{x,w}$. Now suppose that $H \subseteq E^0 \setminus T$, that $T$ is cyclic and
that the cycle $\mu$ with no entrance in $T$ belongs to $\CH$, but that $z \in U(\mu)$.
Arguing as in the preceding paragraph, we see that $\pi_{x,w}(h_{U(\mu)}(z)p_{r(\mu)} - \tau^U_\mu) =
0$. Since $\tau^U_\mu \in J_{H,U}$, we deduce that $p_{r(\mu)} \in J_{H,U} +
\ker\pi_{x,w}$. Since $p_{r(\mu)} \not\in \ker\pi_{x,w}$ by Lemma~\ref{lem:pixz}, we
deduce that $J_{H,U} \not\subseteq \ker\pi_{x,w}$.

For the final statement, observe that $H \subseteq \{v : p_v\in J_{H,U}\}$ by
definition. For the reverse containment, recall that by definition of an ideal pair, each
$U(\mu)$ is a proper subset of $\TT$. So for each $\mu \in \CH$, we can choose $z_\mu \in
\TT \setminus U(\mu)$. By the preceding paragraphs, we have $J_{H,U} \subseteq
I_{[\mu]^0, z_\mu}$ for each $\mu \in \CH$. For each $v \in E^0 \setminus H$ that does
not belong to $[\mu^\infty]^0$ for any $\mu \in \CH$, we can choose an infinite path
$x^v$ in $E^0 \setminus H$ with $r(x^v_1) = v$. This $x^v \not\in [\mu^\infty]$ for $\mu
\in \CH$ because $v$ does not belong to any $[\mu^\infty]^0$. So each $[x^v]^0$ is a
maximal tail contained in the complement of $H$ and the preceding paragraphs show that
$J_{H,U} \subseteq I_{[x^v]^0, 1}$. We now have
\[\textstyle
J_{H,U} \subseteq \Big(\bigcap_\mu I_{[\mu]^0, z_\mu}\Big) \cap
    \Big(\bigcap_v I_{[x^v]^0, 1}\Big).
\]
By construction, the right-hand side does not contain $p_v$ for any $v \not\in H$, and so
we deduce that $v\not\in H$ implies $p_v \not\in J_{H,U}$ as required.
\end{proof}

\begin{proof}[Proof of Theorem~\ref{thm:all ideals}]
To prove the theorem, it suffices to show that the assignment $(H,U) \mapsto J_{H,U}$ is
injective, and then prove statement~\ref{it:ideal->pair}.

The general theory of $C^*$-algebras says that every ideal of a $C^*$-algebra $A$ is
equal to the intersection of all of the primitive ideals that contain it. By definition,
the topology on $\Prim(A)$ is the weakest one in which $\{I \in \Prim(A) : J \subseteq
I\}$ is closed for every ideal $J$ of $A$, and the map which sends $J$ to this closed
subset of $\Prim(A)$ is a bijection. So to prove that $(H,U) \mapsto J_{H,U}$ is
injective, we just have to show that the closed sets $Y_{H,U} := \{I \in \Prim C^*(E) :
J_{H,U} \subseteq I\}$ are distinct for distinct pairs $(H,U)$.

By Lemma~\ref{lem:containment}, we have
\begin{align*}
Y_{H,U} = \{I_{T,z} : {}&T \subseteq E^0 \setminus H\text{ is a maximal tail, and} \\
                &\text{ if $T$ is cyclic and the cycle $\mu$ with no entrance in $T$}\\
                &\text{ also has no entrance in $H$, then $z \not\in U(\mu)$}\}.
\end{align*}

Suppose that $(H_1, U_1)$ and $(H_2, U_2)$ are distinct ideal pairs of $E$. We consider
two cases. First suppose that $H_1 \not= H_2$. Without loss of generality, there exists
$v \in H_1 \setminus H_2$. Since $H_2$ is saturated, there exists $e_1 \in vE^1$ such
that $s(e_1) \not\in H_2$. Since $H_1$ is hereditary, we have $s(e) \in H_1$. Repeating
this argument we obtain edges $e_i \in s(e_{i-1}) E^1$ with $s(e_i) \in H_1 \setminus
H_2$, and hence an infinite path $x$ lying in $(E \setminus EH_1) \setminus (E \setminus
EH_2)$. Now $[x]^0$ is a maximal tail contained in $H_1 \setminus H_2$. If $[x]^0$ is an
aperiodic tail or is a cyclic tail such that the cycle with no entrance in $[x]^0$ has an
entrance in $E \setminus EH_2$, we set $z = 1$. If $[x]^0  = [\mu^\infty]^0$ for some
$\mu \in \Cc(H_2)$, we choose any $z \in \TT \setminus U_2(\mu)$. Then
Lemma~\ref{lem:containment} shows that $I_{[x]^0, z} \in Y_{H_2, U_2} \setminus Y_{H_1,
U_1}$.

Now suppose that $H_1 = H_2$. Then $U_1 \not= U_2$, so we can find $\mu \in \Cc(H_1) =
\Cc(H_2)$ such that $U_1(\mu) \not= U_2(\mu)$. Again without loss of generality, we can
assume that there exists $z \in U_1(\mu) \setminus U_2(\mu)$, and then we have
$I_{[\mu^\infty]^0, z} \in Y_{H_2, U_2} \setminus Y_{H_1, U_1}$. This completes the proof
that the $Y_{H,U}$ are distinct.

It remains to prove~\ref{it:ideal->pair}. Given an ideal $I$, the set $H := H_I$ is a
saturated hereditary set by \cite[Lemma~4.5]{RaeburnCBMS}. Since the ideal $I_H$
generated by $\{p_v : v \in H\}$ is contained in $I$, Lemma~\ref{lem:corner} shows that
$s_\mu + I$ is unitary in $(p_{r(\mu)} + I) C^*(E)/I (p_{r(\mu)} + I)$ for each $\mu \in
C(H)$; so its spectrum is a closed subset of $\TT$, showing that $U_I(\mu)$ is an open
subset of $\TT$. If $\mu,\nu \in C(H)$ with $[\mu^\infty] = [\nu^\infty]$, then
$\mu^\infty = \alpha\nu^\infty$ for some initial segment $\alpha$ of $\mu^\infty$. The
Cuntz--Krieger relations show that $s^*_\alpha s_\mu s_\alpha + I = s_\nu + I$ and
$s_\alpha s_\nu s^*_\alpha + I = s_\mu + I$; so conjugation by $s_\alpha + I$ gives an
isomorphism $C^*(s_\mu) + I \cong C^*(s_\nu) + I$, and in particular
\[
\spec_{(p_{r(\mu)} + I) (C^*(E)/I) (p_{r(\mu)} + I)} (s_\mu + I)
    = \spec_{(p_{r(\nu)} + I) (C^*(E)/I) (p_{r(\nu)} + I)} (s_\nu + I),
\]
giving $U_I(\mu) = U_I(\nu)$. So $(H, U)$ is an ideal pair.

To see that $I = J_{H_I, U_I}$, we first check the containment $\supseteq$. For this, it
suffices to show that every generator of $J_{H_I, U_I}$ belongs to $I$. We have $p_v \in
I$ for all $v \in H_I$ by definition. Fix $\mu \in \Cc(H_I)$; we must show that
$\tau^{U_I}_\mu \in I$. For this, let $I_H$ be the ideal of $C^*(E)$ generated by $\{p_v
: v \in H\}$. Since $I_H$ is contained in both $I$ and $J_{H_I, U_I}$ we just have to
show that $J_{H_I, U_I}/ I_H$ is contained in $I/I_H$. For this, let $\pi : p_{r(\mu)}
C^*(E) p_{r(\mu)} \to C(\TT)$ be the composition of the isomorphism of
Lemma~\ref{lem:corner} with the canonical surjection $p_{r(\mu)} C^*(E) p_{r(\mu)} \to
(p_{r(\mu)} + I_H) (C^*(E)/I_H) (p_{r(\mu)} + I_H)$. Then $\pi(\tau^{U_I}_\mu) =
h_{U_I(\mu)}$ vanishes on $\TT \setminus U_I(\mu)$, which is $\spec_{(p_{r(\mu)} + I)
(C^*(E)/I) (p_{r(\mu)} + I)} (s_\mu + I)$.  Since the quotient map by the image of $I$
under $\pi$ is given by restriction of functions to $\spec_{(p_{r(\mu)} + I) (C^*(E)/I)
(p_{r(\mu)} + I)} (s_\mu + I)$, it follows that $\tau^{U_I}_\mu + I_H \in I/I_H$ as
required.

For the reverse containment, recall that every ideal of $C^*(E)$ is the intersection of
the primitive ideals that contain it, so it suffices to show that if $I_{S,w} \in Y_{H_I,
U_I}$, then $I \subseteq I_{S, w}$. Fix $I_{S, w} \in Y_{H_I, U_I}$. We can express $I$ as
an intersection of primitive ideals and therefore, by Theorem~\ref{thm:prim ideal
catalogue}, we have $I = \bigcap_{(T, z) \in X} I_{T, z}$ for some set $X$ of pairs
consisting of a maximal tail $T$ and an element $z \in \{u^{\Per(T)} : u \in \TT\}$. We
then have
\[\textstyle
v \in H_I \iff p_v \in I \iff p_v \in \bigcap_{(T, z) \in X} I_{T, z}
    \iff v \in \bigcap_{(T, z) \in X} E^0 \setminus T,
\]
and we deduce that $H_I = E^0 \setminus \bigcup_{(T, z) \in X} T$. Since $I_{S, w} \in
Y_{H_I, U_I}$, we have $S \subseteq E^0 \setminus H_I = \bigcup_{(T, z) \in X} T$. So if
$S$ is an aperiodic tail, or is a cyclic tail such that the cycle $\mu$ with no entrance
in $S$ has an entrance in $\bigcup_{(T, z) \in X} T$, then Theorem~\ref{thm:1graph}
immediately gives $I = \bigcap_{(T, z) \in X} I_{T, z} \subseteq I_{S, w}$. So suppose
that $S$ is cyclic, and the cycle $\mu$ with no entrance in $S$ has no entrance in
$\bigcup_{(T, z) \in X} T$. Again using that $I_{S, w} \in Y_{H_I, U_I}$, we see that $w
\not\in U_I(\mu)$. Hence $w \in \spec_{(p_{r(\mu)} + I) (C^*(E)/I) (p_{r(\mu)} + I)}
(s_\mu + I)$. So if $\pi : p_{r(\mu)} C^*(E) p_{r(\mu)} \to C(\TT)$ is the map described
in the preceding paragraph, we have $f(w) = 0$ for all $f$ in $\pi(I) = \bigcap_{(S, z)
\in X} \pi(I_{S, z})$. Each $\pi(I_{S, z})$ is the set of functions that vanishes at $z$,
so we deduce that every function vanishing at every $z$ for which $(S, z) \in X$ also
vanishes at $w$; that is $w \in \overline{\{z : (S, z) \in X\}}$. Now
Theorem~\ref{thm:1graph} again gives $I =  \bigcap_{(T, z) \in X} I_{T, z} \subseteq
I_{S, w}$.
\end{proof}

\begin{rmk}\label{rmk:which primitive}
To see where the primitive ideals of $C^*(E)$ fit into the catalogue of
Theorem~\ref{thm:all ideals}, first let us establish the convention that if $\CH =
\emptyset$, then $\emptyset$ denotes the unique (trivial) function from $\CH$ to the
collection of open subsets of $\TT$, and that if $\CH$ is a singleton, then $\check{z}$
denotes the function on $\CH$ that assigns the value $\TT \setminus \{z\}$ to the unique
element of $\CH$. Now if $T$ is a maximal tail and $z \in \{w^{\Per(T)} : w \in \TT\}$,
then Lemma~\ref{lem:irreps} and the definition of the ideals $J_{H, U}$ show that
\[
    I_{T, z} = \begin{cases}
        J_{E^0 \setminus T, \emptyset} &\text{ if $T$ is aperiodic}\\
        J_{E^0 \setminus T, \check{z}} &\text{ if $T$ is cyclic.}
    \end{cases}
\]
\end{rmk}

\begin{rmk}
The ideal $J_{H,U}$ is gauge invariant (i.e., $\gamma_z(J_{H,U})=J_{H,U}$ for every $z\in\mathbb{T}$) if and only if $U(\mu)=\emptyset$ for every $\mu\in\CH$, in which case $J_{H,U}=I_H$. Thus, we recover from Theorem~\ref{thm:all ideals} the description of the gauge invariant ideals of $C^*(E)$ presented in \cite[Theorem 4.1]{BHRS}.
\end{rmk}

\section{The lattice structure}

To finish off the description of the lattice of ideals of $C^*(E)$, we describe the
complete-lattice structure in terms of ideal pairs.

We define $\preceq$ on the set $\Ii_E$ of ideal pairs for a row-finite graph $E$ with no
sources by
\begin{align*}
(H_1, U_1) \preceq (H_2, U_2)\qquad\Longleftrightarrow\qquad
    H_1 \subseteq H_2\text{ and }& U_1(\mu) \subseteq U_2(\mu)\\
    &\text{ for all } \mu \in \Cc(H_1) \cap \Cc(H_2).
\end{align*}
In the following, given $X \subseteq \TT$, we write $\operatorname{Int}(X)$ for the
interior of $X$.

\begin{thm}\label{thm:lattice}
Let $E$ be a row-finite graph with no sources.
\begin{enumerate}[label=(\arabic*),labelindent=\parindent,leftmargin=*]
\item\label{it:containment} Given ideal pairs $(H_1, U_1)$ and $(H_2, U_2)$ for $E$,
    we have $J_{H_1, U_1} \subseteq J_{H_2, U_2}$ if and only if $(H_1, U_1) \preceq
    (H_2, U_2)$.
\item\label{it:meet} Given a set $K \subseteq \Ii_E$ of ideal pairs for $E$, we have
    $\bigcap_{(H,U) \in K} J_{H,U} = J_{H_K, U_K}$ where $H_K = \bigcap_{(H, U) \in
    K} H$, and $U_K(\mu) = \operatorname{Int}\big(\bigcap_{(H,U) \in K, \mu \in \CH}
    U(\mu)\big)$.
\item\label{it:join} Fix a set $K \subseteq \Ii_E$ of ideal pairs of $E$. Let $A$ be
    the saturated hereditary closure of $\bigcup_{(H,U) \in K} H$. Let $B = \{r(\mu)
    : \mu \in \Cc(A)\text{ and } \bigcup_{(H,U) \in K, \mu \in \CH} U(\mu) = \TT\}$.
    Let $H^K$ be the saturated hereditary closure of $A \cup B$ in $E^0$, and for
    each $\mu \in \Cc(H^K)$, let $U^K(\mu) = \bigcup_{(H,U) \in K, \mu \in \CH}
    U(\mu)$. Then $\clsp\big(\bigcup_{(H, U)\in K} J_{H, U}\big) = J_{H^K, U^K}$.
\end{enumerate}
\end{thm}
\begin{proof}
\ref{it:containment}: First suppose that $(H_1, U_1) \preceq (H_2, U_2)$. We show that
every generator of $J_{H_1, U_1}$ belongs to $J_{H_2, U_2}$. For each $v \in H_1$ we have
$v \in H_2$ and therefore $p_v \in J_{H_2, U_2}$. Suppose that $\mu \in \Cc(H_1)$. If
$r(\mu) \in H_2$, then $p_{r(\mu)} \in J_{H_2, U_2}$ and so $\tau^{U_1}_\mu \in
p_{r(\mu)} C^*(E) p_{r(\mu)}$ belongs to $J_{H_2, U_2}$ as well. So we may suppose that
$r(\mu) \not\in H_2$. Since $H_1 \subseteq H_2$ and since $\mu$ has no entrance in $E^0
\setminus H_1$, it cannot have an entrance in $E^0 \setminus H_2$, so it belongs to
$\Cc(H_2)$. The ideal $I_{H_1}$ generated by $\{p_v : v \in H_1\}$ is contained in both
$J_{H_1, U_1}$ and $J_{H_2, U_2}$. By Lemma~\ref{lem:corner}, we have $(p_{r(\mu)} + I_{H_1})
(C^*(E)/I_{H_1}) (p_{r(\mu)} + I_{H_1}) \cong C(\TT)$ and this isomorphism carries
$\tau^{U_1}_\mu$ to $h_{U_1(\mu)}$ and carries the image of $J_{H_2, U_2}$ to $\{f \in
C(\TT) : f^{-1}(\CC \setminus \{0\}) \subseteq U_2(\mu)\}$. Since $U_1(\mu) \subseteq
U_2(\mu)$, it follows that the image of $\tau^{U_1}_\mu$ in the corner $(p_{r(\mu)} +
I_{H_1}) (C^*(E)/I_{H_1}) (p_{r(\mu)} + I_{H_1})$ belongs to the image of $J_{H_2, U_2}$, and
therefore $\tau^{U_1}_\mu + I_{H_1} \subseteq J_{H_2, U_2}$, giving $\tau^{U_1}_\mu \in
J_{H_2, U_2}$.

Now suppose that $J_{H_1, U_1} \subseteq J_{H_2, U_2}$. The final statement of
Lemma~\ref{lem:containment} shows that $H_1 \subseteq H_2$, so we must show that whenever
$\mu \in \Cc(H_1) \cap \Cc(H_2)$, we have $U_1(\mu) \subseteq U_2(\mu)$.
Theorem~\ref{thm:all ideals}~\ref{it:ideal->pair} shows that
\[
U_i(\mu) = \TT \setminus \spec_{(p_{r(\mu)} + J_{H_i, U_i})
    (C^*(E)/J_{H_i, U_i}) (p_{r(\mu)} + J_{H_i, U_i})} (s_\mu + J_{H_i, U_i}).
\]
Since $J_{H_1, U_1} \subseteq J_{H_2, U_2}$, there is a homomorphism $q : C^*(E)/J_{H_1,
U_1} \to C^*(E)/J_{H_2, U_2}$ that carries $s_\mu + J_{H_1, U_1}$ to $s_\mu + J_{H_2,
U_2}$. In particular, $q$ carries $p_{r(\mu)} + J_{H_1, U_1}$ to $p_{r(\mu)} + J_{H_2,
U_2}$, and so induces a unital homomorphism between the corners determined by these
projections. Since unital homomorphisms decrease spectra, we obtain
\begin{align*}
\spec&_{(p_{r(\mu)} + J_{H_2, U_2}) (C^*(E)/J_{H_2, U_2}) (p_{r(\mu)} + J_{H_2, U_2})} (s_\mu + J_{H_2, U_2}) \\
    &\qquad\subseteq \spec_{(p_{r(\mu)} + J_{H_1, U_1}) (C^*(E)/J_{H_1, U_1}) (p_{r(\mu)} + J_{H_1, U_1})} (s_\mu + J_{H_1, U_1}),
\end{align*}
and hence $U_1(\mu) \subseteq U_2(\mu)$.

\ref{it:meet}: The ideal $\bigcap_{(H, U) \in K} J_{H, U}$ is the largest ideal that is
contained in $J_{H,U}$ for every $(H,U)$ in $K$. Since the map $(H,U) \to J_{H,U}$ is a
bijection carrying $\preceq$ to $\subseteq$, it suffices to show that $(H_K, U_K) \preceq
(H,U)$ for all $(H,U) \in K$, and is maximal with respect to $\preceq$ amongst pairs
$(H'', U'')$ satisfying $(H'', U'') \preceq (H,U)$ for all $(H,U) \in K$. The pair $(H_K,
U_K)$ satisfies $(H_K, U_K) \preceq (H, U)$ for all $(H,U) \in K$ by definition of $H_K$
and $U_K$. Suppose that $(H'', U'') \preceq (H,U)$. Then $H'' \subseteq H$ for all $(H,U)
\in K$, and hence $H'' \subseteq H_K$; and if $\mu \in \Cc(H'') \cap \Cc(H_K)$, and if
$(H,U) \in K$ satisfies $\mu \in \CH$, then $U''(\mu) \subseteq U(\mu)$ because $(H'',
U'') \preceq (H, U)$. So $U''(\mu)$ is an open subset of $\bigcap_{(H,U) \in K, \mu \in
\CH} U(\mu)$, and therefore belongs to $\operatorname{Int}\big(\bigcap_{(H,U) \in K, \mu
\in \CH} U(\mu)\big) = U_K$.

\ref{it:join}: The ideal $\clsp\big(\bigcup_{(H, U) \in K} J_{H, U}\big)$ is the
smallest ideal containing $J_{H,U}$ for every $(H,U)$ in $K$. So as above it suffices to
show that $(H,U) \preceq (H^K, U^K)$ for all $(H,U) \in K$, and that $(H^K, U^K)$ is
minimal with respect to $\preceq$ amongst pairs $(H'', U'')$ satisfying $(H,U) \preceq
(H'', U'')$ for all $(H,U) \in K$. The pair $(H^K, U^K)$ satisfies $(H, U) \preceq (H^K,
U^K)$ for all $(H,U) \in K$ by construction. Suppose that $(H'', U'')$ is another ideal
pair satisfying $(H,U) \preceq (H'', U'')$ for all $(H,U) \in K$. We just have to show
that $(H^K, U^K) \preceq (H'', U'')$. We have $H \subseteq H''$ for every $(H,U) \in K$,
and since $H''$ is saturated and hereditary, it follows that $A \subseteq H''$. If $v \in
B$, then there exists $\mu \in \Cc(A)$ such that $\bigcup_{(H,U) \in K, \mu \in \CH}
U(\mu) = \TT$, and then by compactness of $\TT$, there are finitely many pairs $(H_1,
U_1), \dots, (H_n, U_n) \in K$ such that $\mu \in \Cc(H_i)$ for each $i$, and
$\bigcup^n_{i=1} U(\mu) = \TT$. Choose a partition of unity $\{f_1, \dots, f_n\} \in
C(\TT)$ subordinate to the $U_i$. Let $I_A$ be the ideal of $C^*(E)$ generated by $\{p_v
: v \in A\}$. Then each $f_i$ belongs to the image of $p_{r(\mu)} J_{(H_i, U_i)}
p_{r(\mu)}$ under the isomorphism of Lemma~\ref{lem:corner}, and so $1 = \sum_i f_i$
belongs to the image of $\sum^n_{i=1} p_{r(\mu)} J_{H_i, U_i} p_{r(\mu)}$. Since each
$(H_i, U_i) \preceq (H'', K'')$, it follows that $1$ belongs to the image of $J_{(H'',
K'')}$. But the preimage of $1$ is $p_{r(\mu)} + I_A$, and we deduce that $p_{r(\mu)} \in
J_{(H'', K'')}$. The final statement of Lemma~\ref{lem:containment} therefore implies
that $v \in H''$. So $A \cup B \subseteq H''$, and since $H''$ is saturated and
hereditary, it follows that $H^K \subseteq H''$. Now suppose that $\mu \in \Cc(H^K) \cap
\Cc(H'')$. For each $z \in U^K(\mu)$, there exists $(H,U) \in K$ such that $\mu \in \CH$ and
$z \in U(\mu)$. Since $(H,U) \preceq (H'', U'')$ and $\mu \in \Cc(H'') \cap \Cc(H)$, we
deduce that $z \in U''(\mu)$. So $U^K(\mu) \subseteq U''(\mu)$. So we have $(H^K, U^K)
\preceq (H'', U'')$ as required.
\end{proof}

\end{document}